\documentclass[10pt,reqno]{amsart}

\usepackage{amsthm,amsmath,amsfonts,amssymb,amsxtra,appendix,bookmark,dsfont,latexsym,epsfig,graphicx,stix}

\makeatletter
\usepackage{hyperref}
\usepackage{delarray,color}
\usepackage{euscript}

\usepackage{relsize}
\usepackage{enumitem}
\usepackage[dvipsnames]{xcolor}

\usepackage[T1]{fontenc}

\usepackage[belowskip=2pt,aboveskip=0pt]{caption}
\usepackage{booktabs}
\usepackage{hyperref}

\makeatletter
\def\@tocline#1#2#3#4#5#6#7{\relax
  \ifnum #1>\c@tocdepth 
  \else
    \par \addpenalty\@secpenalty\addvspace{#2}%
    \begingroup \hyphenpenalty\@M
    \@ifempty{#4}{%
      \@tempdima\csname r@tocindent\number#1\endcsname\relax
    }{%
      \@tempdima#4\relax
    }%
    \parindent\z@ \leftskip#3\relax \advance\leftskip\@tempdima\relax
    \rightskip\@pnumwidth plus4em \parfillskip-\@pnumwidth
    #5\leavevmode\hskip-\@tempdima
      \ifcase #1
       \or\or \hskip 1em \or \hskip 2em \else \hskip 3em \fi%
      #6\nobreak\relax
      \dotfill
      \hbox to\@pnumwidth{\@tocpagenum{#7}}
    \par
    \nobreak
    \endgroup
  \fi}
\makeatother

\newcommand{\N}{\mathbb N} 
\newcommand{\R}{\mathbb R} 

\newcommand{\Ec}{\mathcal E}

\newcommand{\Lc}{\mathcal L}

\newcommand{\vareps}{\varepsilon}
\newcommand{\eps}{\epsilon}

\newcommand{\dx}{{\rm d} x}

\def\({\left(}
\def\){\right)}
\def\<{\left\langle}
\def\>{\right\rangle}

\DeclareMathOperator*{\nablax}{{\nabla_{x^\perp}}}

\DeclareMathOperator*{\ima}{Im}
\DeclareMathOperator*{\rea}{Re}

\DeclareMathOperator*{\Tr}{Tr}

\numberwithin{equation}{section}
\newtheorem{theorem}{Theorem}[section]
\newtheorem{lemma}[theorem]{Lemma}

\newtheorem{corollary}[theorem]{Corollary}

\newtheoremstyle{remarkstyle}
{}{}{
}{ }{\bfseries}{.}{ }{\thmname{#1}\thmnumber{ #2}\thmnote{ (#3)}}
\theoremstyle{remarkstyle}
\newtheorem{remark}{Remark}[section]

\title[Blow-up 2D RBEC]{Blow-up of 2D attractive Bose--Einstein condensates \\ at the critical rotational speed} 
\author[V. D. Dinh, D.-T. Nguyen, and N. Rougerie]{Van Duong Dinh, Dinh-Thi Nguyen, and Nicolas Rougerie}
\address[V. D. Dinh]{Ecole Normale Sup\'erieure de Lyon \& CNRS, UMPA (UMR 5669), Lyon, France}
\email{contact@duondinh.com}
\address[D.-T. Nguyen]{Ecole Normale Sup\'erieure de Lyon \& CNRS, UMPA (UMR 5669), Lyon, France}
\email{dinh.nguyen@ens-lyon.fr}
\address[N. Rougerie]{Ecole Normale Sup\'erieure de Lyon \& CNRS, UMPA (UMR 5669), Lyon, France}
\email{nicolas.rougerie@ens-lyon.fr}

\date{August 2022}

\begin{document}

\begin{abstract}
We study the ground states of a 2D focusing non-linear Schr\"odinger equation with rotation and harmonic trapping. When the strength of the interaction approaches a critical value from below, the system collapses to a profile obtained from the optimizer of a Gagliardo--Nirenberg interpolation inequality. This was established before in the case of fixed rotation frequency. We extend the result to rotation frequencies approaching, or even equal to, the critical frequency at which the centrifugal force compensates the trap. We prove that the blow-up scenario is to leading order unaffected by such a strong deconfinement mechanism. In particular the blow-up profile remains independent of the rotation frequency.  
\end{abstract}

	\maketitle

	\tableofcontents
		
	\section{Introduction}
	\label{S1}
	\setcounter{equation}{0}
	
	Bose--Einstein condensates~\cite{CorWie-nobel,Ketterle-nobel} form a remarkable phase of matter where quantum effects can be spectacularly observed on a mesoscopic scale. Indeed, a single quantum wave-function being macroscopically occupied, its' quantum coherence becomes accessible e.g. to imaging techniques. The flexibility of modern experiments with dilute atomic gases are also remarkable~\cite{Aftalion-07,BloDalZwe-08,DalGerJuzOhb-11,Cooper-08,PitStr-03,PetSmi-01}, allowing to access reduced dimensionalities (2D or even 1D), to tune the interactions (allowing for repulsion or attraction between particles) and to mimic external magnetic fields either by rotation or by coupling internal degrees of freedom to optical fields. 
	
	In this note we consider such a combination of effects. Namely we are interested in 2D attractive BECs, where the contact interactions will destabilize the gas towards collapse if they are too strong. The resulting collapse of ground states~\cite{GuoSei-13} turns out to be unaffected by the addition of a moderate rotation of the gas~\cite{LewNamRou-17-proc} (see also~\cite{EycRou-18} for dipolar gases). A fast rotation may however destabilize the gas towards expansion, because the centrifugal force fights the confining potential. These two effects might compete, but we prove that the instability towards collapse always dominates, leading to a blow-up scenario independent of the rotation frequency. This answers a question raised in~\cite[Remark~2.2]{LewNamRou-17-proc}.  
	
	We shall consider the minimization problem
	\begin{align} \label{ener-mini}
	E^{\rm NLS}_{\Omega,a}:= \inf \left\{ \Ec^{\rm NLS}_{\Omega,a}(\phi) : \phi \in X(\R^2) :\|\phi\|_{L^2} =1\right\},
	\end{align}
	where $\Ec^{\rm NLS}_{\Omega, a}$ is the nonlinear Schr\"odinger (NLS) energy functional with attractive interactions
	\begin{align*}
	\Ec^{\rm NLS}_{\Omega, a}(\phi) &= \int_{\R^2} |\nabla \phi(x)|^2 \dx + \int_{\R^2} |x|^2 |\phi(x)|^2 \dx +2 \Omega \<L\phi, \phi\> -\frac{a}{2} \int_{\R^2} |\phi(x)|^4 \dx \\
	&= \int_{\R^2} |(-i\nabla + \Omega x^\perp) \phi(x))|^2 \dx + (1-\Omega^2) \int_{\R^2} |x|^2 |\phi(x|^2 \dx -\frac{a}{2} \int_{\R^2} |\phi(x)|^4 \dx.
	\end{align*}
	Here $a>0$ describes the strength of interactions, $\Omega \geq 0$ is the rotation frequency and $L=-ix\wedge \nabla = i(x_2 \partial_1-x_1 \partial_2)$ the angular momentum operator. The space $X(\R^2)$ in \eqref{ener-mini} is a functional space in which the energy functional $\Ec^{\rm NLS}_{\Omega, a}$ is well-defined, see below.
	
	In the case of high rotational speed $\Omega>1$, it was proved in \cite{BaoWanMar-05} that there are no ground states for $E^{\rm NLS}_{\Omega, a}$ for all $a>0$. Indeed, when the rotational speed is larger than the trapping frequency, the centrifugal force caused by the rotation is stronger than the centripetal force created by the harmonic trap and the gas flies apart. On the other hand, the condensate remains stable when $\Omega<1$. In this case, one can prove the norm equivalence
	\begin{equation}\label{norm-equavalent}
	\|\nabla \phi\|^2_{L^2} + \|x\phi\|^2_{L^2} + 2 \Omega \<L\phi, \phi\> \simeq \|\nabla \phi\|^2_{L^2} + \|x\phi\|^2_{L^2}.
	\end{equation}
	It is then clear that the energy functional is well-defined on the weighted Sobolev space
	\[
	\Sigma(\R^2):= H^1(\R^2) \cap L^2(\R^2, |x|^2 \dx),
	\]
	and hence one can take $X(\R^2) \equiv \Sigma(\R^2)$. Using the compact embedding $\Sigma(\R^2) \subset L^r(\R^2)$ for all $r \in [2,\infty)$, one can easily show the existence of a ground state for $E^{\rm NLS}_{\Omega,a}$ with $0<a<a_*$ (see e.g., \cite{GuoSei-13} in the case $\Omega = 0$). Here $a_*=\|Q\|_{L^2}^2$ with $Q$ the unique (up to translations) positive solution of the elliptic equation
	\begin{align} \label{Q}
	-\Delta Q + Q - Q^3=0 \quad \text{in} \quad \R^2.
	\end{align}
	The constant $a_*$ also appears in the sharp Gagliardo--Nirenberg inequality
	\begin{align} \label{GN-ineq}
	\frac{a_*}{2} \int_{\R^2}|\phi(x)|^4 \dx \leq  \left(\int_{\R^2} |\nabla \phi(x)|^2 \dx\right) \left( \int_{\R^2} |\phi(x)|^2 \dx\right), \quad \forall \phi \in H^1(\R^2).
	\end{align}
	The case of critical rotational speed $\Omega=1$ is special. The situation becomes more subtle since the centrifugal force caused by the rotation is exactly compensated by the harmonic trap. In particular, the norm equivalence \eqref{norm-equavalent} is no longer available. Thus working on $\Sigma(\R^2)$ does not help to find ground states for $E^{\rm NLS}_{1,a}$. In this case, we study the minimization \eqref{ener-mini} on a larger functional space of magnetic Sobolev functions, namely
	\[
	H^1_{x^\perp}(\R^2):= \left\{\phi \in L^2(\R^2) : (-i\nabla +x^\perp) \phi \in L^2(\R^2)\right\},
	\]
	hence we set $X(\R^2)= H^1_{x^\perp}(\R^2)$ when $\Omega = 1$. By making use of the concentration-compactness argument adapted to magnetic Sobolev spaces (see e.g., \cite{EstLio-89}), it was proved in \cite{Dinh-22,GuoLuoPen-21} that $E^{\rm NLS}_{1,a}$ has at least one ground state provided that $0<a<a_*$. By \eqref{GN-ineq} and the diamagnetic inequality (see e.g., \cite[Theorem 7.21]{LieLos-01})
	\begin{equation} \label{dia-ineq}
	|\nabla |\phi|(x)| \leq |(-i\nabla + x^\perp) \phi(x)|, \quad \mbox{a.e } x\in \R^2, \quad \forall \phi \in H^1_{x^\perp}(\R^2)
	\end{equation}
	we also have the following magnetic Gagliardo--Nirenberg inequality 
	\begin{align} \label{mag-GN-ineq}
	\frac{a_*}{2} \int_{\R^2}|\phi(x)|^4 \dx \leq  \left(\int_{\R^2} |(-i\nabla + x^\perp) \phi(x)|^2 \dx\right) \left( \int_{\R^2} |\phi(x)|^2 \dx\right), \quad \forall \phi \in H^1_{x^\perp}(\R^2).
	\end{align}
	The main difference between \eqref{GN-ineq} and \eqref{mag-GN-ineq} is that there is no optimizer for \eqref{mag-GN-ineq} while $Q$ in \eqref{Q} is the unique (up to translations and dilations) optimizer for \eqref{GN-ineq}.
	
	\subsection{Collapse in NLS theory}
	
	In the sequel we are interested in the blow-up behavior of ground states for $E^{\rm NLS}_{\Omega,a}$ when $a$ approaches $a_*$. Our first result concerns the blow-up limit with the critical rotation speed $\Omega=1$. 
	
	\begin{theorem}[\textbf{Collapse of NLS ground states at critical rotational speed}]\label{theo-blow-nls}\mbox{}\\
		We have, as $a\nearrow a_*$,
		\begin{equation}\label{nls-energy}
		E^{\rm NLS}_{1,a} = (a_*-a)^{1/2}\left(2\frac{\|xQ_0\|_{L^2}}{a_*^{1/2}} + o(1)\right)
		\end{equation}
		where $Q_0=\|Q\|_{L^2}^{-1}Q$. In addition, for any sequence $\{a_n\}_{n}$ satisfying $a_n \nearrow a_*$ and any sequence of ground state $\phi_n$ for $E^{\rm NLS}_{1,a_n}$, there exist a sequence $\{\theta_n\}_n \subset [0,2\pi)$ and a sequence $\{x_n\}_n \subset \R^2$ such that the following convergence holds strongly in $H^1 \cap L^\infty(\R^2)$:
		\begin{equation}\label{nls-gs-critical}
		\lim_{n\to\infty}\frac{(a_*-a_n)^{1/4}}{a_*^{1/4}\|xQ_0\|_{L^2}^{1/2}} \phi_n \(\frac{(a_*-a_n)^{1/4}}{a_*^{1/4}\|xQ_0\|_{L^2}^{1/2}}x + x_n\) \exp\left(i \frac{(a_*-a_n)^{1/4}}{a_*^{1/4}\|xQ_0\|_{L^2}^{1/2}} x_n^\perp \cdot x + i\theta_n\right) = Q_0(x).
		\end{equation}
	\end{theorem}
	
	As an application of this result, we have the following blow-up behavior of ground states when $\Omega \nearrow 1$ and $a \nearrow a^*$ at the same time.
	\begin{corollary}[\textbf{Collapse at subcritical rotational speed}] \label{coro-blow-nls}\mbox{}\\
		For any sequence $\{\Omega_n\}_n, \{a_n\}_n$ satisfying $\Omega_n \nearrow 1$ and $a_n \nearrow a_*$, and any ground state $\phi_n$ for $E^{\rm NLS}_{\Omega_n,a_n}$, there exists a sequence $\{\theta_n\}_n\subset [0,2\pi)$ such that the following convergence holds strongly in $H^1 \cap L^\infty(\R^2)$:
		\begin{equation}\label{nls-gs}
		\lim_{n\to\infty}\frac{(a_*-a_n)^{1/4}}{a_*^{1/4}\|xQ_0\|_{L^2}^{1/2}} \phi_n \(\frac{(a_*-a_n)^{1/4}}{a_*^{1/4}\|xQ_0\|_{L^2}^{1/2}}x\) e^{i\theta_n} = Q_0(x).
		\end{equation}
	\end{corollary}
	
	\begin{remark}\mbox{}\\
		\begin{enumerate}
			\item[1.] The convergences of energy and of ground states were proved by Guo and Seiringer \cite{GuoSei-13} when $\Omega=0$. These convergences were extended to the case $0<\Omega<1$ fixed by Lewin, Nam, and the third author \cite{LewNamRou-17-proc} (see also further work in~\cite{GuoZenZho-16,GuoLinWei-17,DenGuoLu-15}). In~\cite{GuoLuoYan-20} it is even proved that a fixed rotation rate has no effect at any order. Theorem \ref{theo-blow-nls} shows that the energy convergence found remains valid in the case of critical rotational speed $\Omega=1$, at least at leading order. This is noteworthy because the trapping potential, which sets the length-scale of the blow-up behavior, is compensated by the centrifugal force.
			\item[2.] The convergence of ground states however has to be stated differently from~\cite{GuoSei-13,LewNamRou-17-proc}. The model is translation-invariant for $\Omega = 1$ and thus ground states converge only  modulo a magnetic translation (namely, a translation decorated by the suitable phase making it commute with the magnetic Laplacian see e.g.~\cite{Perice-22} and references therein). 
			\item[3.] The only effect of the magnetic/rotation field is to set the blow-up length-scale (see the sketch of proof below). This is comparable to the positive particle mass $m>0$ in the Hartree-type and Thomas--Fermi-type models of stars \cite{GuoZen-17,Nguyen-17, Nguyen-19-RMP, Nguyen-19-JMP, Nguyen-19-CVPDE}.
			\item[4.] Our blow-up result, when $\Omega \nearrow 1$ at the same time as $a \nearrow a_*$, answers a question raised in \cite[Remark 2.2]{LewNamRou-17-proc}. In this situation, although the centrifugal force almost compensates the trapping potential, the small residual trapping favors blow-up at the center of the trap. Hence there is no need for a magnetic translation and the ground state convergence is completely similar to the case $0\leq \Omega <1$ fixed. 
		\end{enumerate}
	\end{remark}
	
	Let us briefly describe the strategy of the proof. To prove Theorem \ref{theo-blow-nls}, we first show, by contraction, that the sequence of ground states $\{\phi_n\}_n$ for $E^{\rm NLS}_{1,a_n}$ blows up in the sense that
	\begin{align} \label{length}
	\vareps_n:= \|\nabla|\phi_n|\|^{-1}_{L^2} \to 0 \text{ as } n \to \infty.
	\end{align}
	The blow-up length is then set by $\vareps_n$ (whose precise asymptotic behavior is not known at this point) and we shall show that
	\[
	\varphi_n(x):=\vareps_n \phi_n(\vareps_n x+x_n) e^{i\vareps_n x_n^\perp \cdot x+i\theta_n}  \to Q_0(x)
	\]
	strongly in $H^1(\R^2)$, i.e. there is convergence modulo a magnetic translation of vector $\{x_n\}_n \subset \R^2$ and the choice of a constant phase $\{\theta_n\}_n \subset [0,2\pi)$. To prove this, we rely on a property of the Lagrange multiplier associated to $\phi_n$ together with the local boundedness of sub-solutions obtained by analyzing the corresponding Euler--Lagrange equation. Thanks to the non-degeneracy of $Q$, we then prove that the imaginary part of $\varphi_n$ is small in $H^1$-norm. This implies that the rotation acts on $\varphi_n$ only as a quadratic external potential. This effectively sets a length-scale, and we next prove by matching energy lower and upper bounds that the blow-up length behaves like 
	$$\frac{(a_*-a_n)^{1/4}}{a_*^{1/4}\|xQ_0\|_{L^2}^{1/2}}.$$
	Hence we obtain the energy convergence \eqref{nls-energy}. Finally, the $L^\infty$-convergence of ground states follows from $H^1$-convergence and $H^2$-boundedness deduced from the variational equation.
	
	To prove Corollary \ref{coro-blow-nls}, we first use an energy argument to show that $E^{\rm NLS}_{\Omega_n, a_n}$ has the same asymptotic behavior as for $\Omega=0,1$. By taking a sequence of ground states for $E^{\rm NLS}_{1,a_n}$ and choosing a suitable trial state for $E^{\rm NLS}_{\Omega_n,a_n}$, we prove that a ground state for $E^{\rm NLS}_{\Omega_n,a_n}$ is an approximate ground state for $E^{\rm NLS}_{0,a_n}$. At this point, the conclusion follows directly from a result proved in \cite[Section~3]{LewNamRou-17-proc}. 
	
	\subsection{Collapse in the mean-field limit}

	The focusing NLS functional~\eqref{ener-mini} is commonly used to predict the collapse of an attractive system, but it should be seen as an effective, mean-field model~\cite{Rougerie-EMS}. It is of interest to see whether the mean-fied and blow-up limits can be commuted as in~\cite{LewNamRou-17-proc}. Based on Theorem \ref{theo-blow-nls} and Corollary \ref{coro-blow-nls}, we give a positive answer to this question, starting from many-body quantum mechanics.
	
	In many-body quantum mechanics, a Bose gas with an attractive interaction is described by the $N$-particle Hamiltonian
	\begin{equation}\label{eq:HN}
	H_{a,N} = \sum_{j=1} ^N  \left( -\Delta_{x_j} + |x_j|^2-2\Omega L_{x_j}\right) - \frac{a}{N-1} \sum_{1\leq i<j \leq N}w_N(x_i-x_j),
	\end{equation}
	acting on $\mathfrak{H}^N:=L_{\rm sym}^2((\R^{2})^N)$. As is customary \cite{Rougerie-EMS}, the two-body interaction $w_N$ is chosen in the form 
	\begin{align} \label{eq:assumption-wN}
	w_N(x)=N^{2\beta} w(N^\beta x)
	\end{align}
	for a fixed parameter $\beta>0$ and a fixed function $w$ satisfying
	\begin{align}
	\label{eq:assumption-w1} 
	w(x)=w(-x) \geq 0, \quad (1+|x|)w,~ \hat w  \in L^1(\R^2),  \quad   \int_{\R^2} w(x) \dx= 1.
	\end{align}
	We are interested in the large-$N$ behavior of the ground state energy per particle of $H_{a,N}$, namely
	\begin{equation}\label{eq:GS ener many}
	E^{\rm QM}_{\Omega,a}(N) := N^{-1}\inf_{\Phi_N\in \mathfrak{H}^N, \|\Phi_N\|=1} \< \Phi_N, H_{a,N} \Phi_N \>,
	\end{equation}
	and the associated eigenstates of $H_{a,N}$. When $\Omega=1$, the Hamiltonian $H_N$ is magnetic translation invariant so it has actually no $L^2$-eigenfunction. In the following, we therefore assume that $0\leq \Omega<1$ and $0<a<a_*$. We will consider the limit where $a=a_N \nearrow a_*$ at the same time as $\Omega = \Omega_N \nearrow 1$ when $N\to \infty$. In that case, the NLS ground states blow up at the origin to the function $Q_0$, as showed in Corollary \ref{coro-blow-nls}. We will prove that the many-body ground states condense fully on $Q_0$. As usual, the convergence of ground states is formulated using $k$-particles reduced density matrices, defined for any $\Phi_N \in\mathfrak{H}^N$ by a partial trace 
	\[
	\gamma_{\Phi_N}^{(k)}:= \Tr_{k+1\to N} |\Phi_N \rangle \langle \Phi_N|.
	\]
	Equivalently, $\gamma_{\Phi_N}^{(k)}$ is the trace class operator on $\mathfrak{H}^k$ with kernel 
	\[
	\gamma_{\Phi_N}^{(k)} (x_1,...,x_k; y_1,...,y_k)= \int_{\R^{2(N-k)}} \Phi_N(x_1,...,x_k,Z) \overline{\Phi_N(y_1,...,y_k,Z)} {\rm d}Z.
	\]  
	Bose--Einstein condensation is properly expressed by the convergence in trace norm
	\[
	\lim_{N\to \infty} \Tr \Big| \gamma_{\Phi_N}^{(k)} - |\phi^{\otimes k}\rangle \langle \phi^{\otimes k}| \Big| =0, \quad \forall k\in \mathbb{N}.
	\] 
	We have the following result.
	
	\begin{theorem}[\textbf{Collapse and condensation of the many-body ground states}]\label{thm:many-body}\mbox{}\\
		Let $0<\beta<1/2$ be fixed and $a=a_N=a_*-N^{-\alpha}$ with 
		\[
		0<\alpha< \min\left\{\frac{4}{5}\beta, 2(1-2\beta)\right\}.
		\]
		Then for every $0\leq\Omega < 1$ we have, as $N\to\infty$,
		\begin{equation}\label{eq:CV_energy}
		E^{\rm QM}_{\Omega,a_N}(N) = E^{\rm NLS}_{\Omega,a_N} + o\big(E^{\rm NLS}_{\Omega,a_N}\big) = (a_*-a_N)^{1/2}\left(2\frac{\|xQ_0\|_{L^2}}{a_*^{1/2}} +o(1)\right).
		\end{equation}
		Assume in addition that $\Omega=\Omega_N = 1- N^{-\nu}$ with
		\[
		0 < \nu < \min\left\{1-2\beta - \frac{\alpha}{2},\beta - \frac{5\alpha}{4}\right\}.
		\]
		Let $\Phi_N$ be a ground state for $E^{\rm QM}_{\Omega, a_N}(N)$. Then we have 
		\begin{align} \label{eq:thm-BEC}
		\lim_{N \to \infty} \Tr\Big| \gamma_{\Phi_{N} }^{(k)} -  |Q_N^{\otimes k} \rangle \langle Q_{N}^{\otimes k}| \Big|=0
		\end{align}
		for all $k\in \mathbb{N}$, where
		\[
		Q_N(x) = \frac{a_*^{1/4}\|xQ_0\|_{L^2}^{1/2}}{(a_*-a_N)^{1/4}} Q_0 \left( \frac{a_*^{1/4}\|xQ_0\|_{L^2}^{1/2}}{(a_*-a_N)^{1/4}}x\right).
		\]	
	\end{theorem}
	
	\begin{remark}
		This shows that a result found in~\cite{LewNamRou-17-proc} remains valid when $\Omega \nearrow 1$ slower than $a \nearrow a_*$. The method is the same as in~\cite{LewNamRou-17-proc}. The energy estimates do not depend on the rotation parameter. In fact, we also obtain \eqref{eq:CV_energy} for $\Omega = 1$. Furthermore, the convergence of the many-body ground states follows from that of the approximate NLS ground states. In the case $\Omega_N \nearrow 1$, under the additional assumption on the convergence speed of $\Omega_N$ in Theorem \ref{thm:many-body}, we check that the approximate NLS ground states for $E^{\rm NLS}_{\Omega_N,a_N}$ is still one for $E^{\rm NLS}_{0,a_N}$.
	\end{remark}

	\section{Collapse of the NLS ground states}
	\label{S2}
	\setcounter{equation}{0}
	In this section we study the limiting behavior of ground states for \eqref{ener-mini} when $a$ approaches $a_*$ from below. We first deal with the critical speed $\Omega=1$. The case $\Omega \nearrow 1$ will be given in the end of this section.
	
	\subsection{Collapse with a critical speed}
	Let us consider the case $\Omega=1$. For simplicity, we denote $\nablax:= -i\nabla +x^\perp$. Let us start by recalling some useful facts.
	
	\begin{lemma}[\textbf{$L^2$-bound}]\label{lem-L2}\mbox{}\\
		We have
		\[
		2\|\phi\|^2_{L^2} \leq \|\nablax \phi\|^2_{L^2}, \quad \forall \phi \in H^1_{x^\perp}(\R^2)
		\]
		with equality achieved e.g. by $\phi(x)=\sqrt{\frac{1}{\pi}} e^{-\frac{|x|^2}{2}}$.
	\end{lemma}
	
	This is a consquence of Landau's well-known diagonalization of $\nablax^2$.
	
	\begin{lemma}[\textbf{Compactness modulo translations}]\label{lem-comp}\mbox{}\\
		Let $\{\phi_n\}_n$ be a sequence of functions satisfying
		\[
		\inf_{n\geq 1} \|\phi_n\|_{L^4} \geq C.
		\]
		for some positive constant $C>0$. We have the following weak convergences:
		\begin{itemize}
			\item If $\sup_{n\geq 1} \|\phi_n\|_{H^1} <\infty$, then there exist $\phi \in H^1 (\R^2) \backslash \{0\}$ and a sequence $\{x_n\}_n \subset \R^2$ such that up to a subsequence,
			\[
			\phi_n(x+x_n) \rightharpoonup \phi(x) \text{ weakly in } H^1(\R^2).
			\]
			\item If $\sup_{n\geq 1} \|\phi_n\|_{H^1_{x^\perp}} <\infty$, then there exist $\tilde{\phi} \in H^1_{x^\perp}(\R^2) \backslash \{0\}$ and a sequence $\{y_n\}_n \subset \R^2$ such that up to a subsequence,
			\[
			e^{i y_n^\perp \cdot x} \phi_n(x+y_n) \rightharpoonup \tilde{\phi}(x) \text{ weakly in } H^1_{x^\perp}(\R^2).
			\]
		\end{itemize}
	\end{lemma}
	\begin{proof}
		The proof of this Lemma can be found in \cite{Lieb-83} and \cite{Dinh-22}.
	\end{proof}

	\begin{lemma}[\textbf{Energy upper bound}] \label{lem-limi-an}\mbox{}\\
		Let $\{a_n\}_n$ be a positive sequence satisfying $a_n \nearrow a_*$ as $n\rightarrow \infty$. Then, for every $0\leq \Omega \leq 1$, we have
		\[
		\lim_{n \to \infty} E^{\rm NLS}_{\Omega,a_n} = E^{\rm NLS}_{\Omega,a_*}=0.
		\]
		More precisely,
		\begin{align}\label{limsup}
		\limsup_{n \to \infty} \frac{E^{\rm NLS}_{1,a_n}}{(a_*-a_n)^{1/2}} \leq 2\frac{\|xQ_0\|_{L^2}}{a_*^{1/2}}.
		\end{align}
	\end{lemma}
	
	\begin{proof}
		It is obvious that $E^{\rm NLS}_{\Omega,a_n} \geq 0$, by the magnetic Gagliardo--Nirenberg inequality \eqref{mag-GN-ineq}. On the other hand, let $Q$ be the unique positive radial solution of \eqref{Q}. By Pohozaev's identity, we have
		\[
		\|\nabla Q\|^2_{L^2} = \frac{1}{2} \|Q\|^4_{L^4} = \|Q\|^2_{L^2}=a_*.
		\]
		Denote $Q_0 = \|Q\|_{L^2}^{-1}Q$. Then
		\[
		\|\nabla Q_0\|^2_{L^2}=\frac{a_*}{2}\|Q_0\|^4_{L^4}=\|Q_0\|^2_{L^2}=1
		\]
		By the variational principle, we have
		\begin{equation}\label{est-E-a}
		E^{\rm NLS}_{\Omega,a_n} \leq \Ec^{\rm NLS}_{\Omega,a_n}(\lambda Q_0(\lambda \cdot)) = \lambda^2  \left(1-\frac{a_n}{a_*}\right) + \lambda^{-2} \|xQ_0\|^2_{L^2}
		\end{equation}
		for all $\lambda>0$. Here we have used the fact that $\<L(\lambda Q_0(\lambda \cdot)), \lambda Q_0(\lambda \cdot)\> =0$ since $Q_0$ is real-valued. Optimizing over $\lambda$, we get
		\begin{align} \label{uppe-boun-Ea}
		E^{\rm NLS}_{\Omega,a_n} \leq 2\frac{\|xQ_0\|_{L^2}}{a_*^{1/2}} (a_*-a_n)^{1/2}
		\end{align}
		which implies \eqref{limsup} and also $\limsup_{n \to \infty} E^{\rm NLS}_{\Omega,a_n}\leq 0$.
	\end{proof}

	\begin{lemma}[\textbf{Blow-up}] \label{lem-blow}\mbox{}\\
		Let $\{a_n\}_n$ be a positive sequence such that $a_n \nearrow a_*$ as $n\rightarrow \infty$ and $\phi_n$ be a ground state for $E^{\rm NLS}_{1,a_n}$. Then $\{\phi_n\}_n$ blows up both in $H^1_{x^\perp}(\R^2)$ and in $H^1(\R^2)$ in the sense that
		\[
		\lim_{n\rightarrow \infty}\|\nablax \phi_n\|_{L^2} = \lim_{n\rightarrow \infty}\|\nabla \phi_n\|_{L^2} = \lim_{n\rightarrow \infty}\|\nabla |\phi_n|\|_{L^2}= +\infty.
		\]
	\end{lemma}
	
	\begin{proof}
		We first show that $\{\phi_n\}_n$ blows up in $H^1_{x^\perp}(\R^2)$. Assume for contradiction that 
		\begin{align} \label{boun-phi-n}
		\sup_{n\geq 1}\|\nablax \phi_n\|^2_{L^2}<\infty.
		\end{align}
		In particular, $\{\phi_n\}_n$ is then a bounded sequence in $H^1_{x^\perp}(\R^2)$. Observe that there exists $C>0$ such that
		\[
		\liminf_{n\rightarrow \infty} \|\phi_n\|_{L^4}\geq C
		\]
		since otherwise, we have 
		\[
		\lim_{n\to \infty} E^{\rm NLS}_{1,a_n}\geq \lim_{n\to \infty} \|\nablax \phi_n\|^2_{L^2} \geq 2,
		\]
		where the last inequality is due to Lemma \ref{lem-L2}. This, however, is not possible (see Lemma \ref{lem-limi-an}). Thus, by Lemma \ref{lem-comp}, there exist $\phi \in H^1_{x^\perp}(\R^2) \backslash \{0\}$ and a sequence $\{x_n\}_n\subset \R^2$ such that
		\[
		\tilde{\phi}_n(x):= e^{i x_n^\perp \cdot x} \phi_n(x+x_n) \rightharpoonup \phi \text{ weakly in } H^1_{x^\perp}(\R^2).
		\]
		We claim that $\|\phi\|^2_{L^2}=1$. Indeed, we have
		\[
		0<\|\phi\|^2_{L^2} \leq \liminf_{n\rightarrow \infty} \|\tilde{\phi}_n\|^2_{L^2} = \liminf_{n\rightarrow \infty} \|\phi_n\|^2_{L^2}= 1.
		\]
		If $\|\phi\|^2_{L^2}<1$, then by the magnetic translation invariance, we have 
		\begin{align}
		E^{\rm NLS}_{1,a_n} = \Ec^{\rm NLS}_{1,a_n}(\phi_n) = \Ec^{\rm NLS}_{1,a_n}(\tilde{\phi}_n) \geq \Ec^{\rm NLS}_{1,a_*}(\tilde{\phi}_n) = \Ec^{\rm NLS}_{1,a_*}(\phi) + \Ec^{\rm NLS}_{1,a_*}(\tilde{\phi}_n-\phi) + o(1). \label{ener-decom}
		\end{align}
		Here we have used the weak convergence  $\tilde{\phi}_n(x) \rightharpoonup \phi$ and the fact that $\|\tilde{\phi}_n\|_{L^4}$ is bounded uniformly, by \eqref{mag-GN-ineq} and \eqref{boun-phi-n}. Again, \eqref{mag-GN-ineq} implies that 
		$$
		\Ec^{\rm NLS}_{1,a_*}(\tilde{\phi}_n-\phi) \geq 0.
		$$
		Furthermore,
		\[
		\Ec^{\rm NLS}_{1,a_*}(\phi) = \|\phi\|^2_{L^2} \Ec^{\rm NLS}_{1,a_*} \left( \frac{\phi}{\|\phi\|_{L^2}}\right) + \frac{a_*}{2}\left( \frac{1}{\|\phi\|^2_{L^2}}-1\right)\|\phi\|^4_{L^4} > 0
		\]
		since $0<\|\phi\|_{L^2}<1$. This contradicts the fact that $E^{\rm NLS}_{1,a_n} \to 0$ as $n\to\infty$, by Lemma \ref{lem-limi-an}. Therefore, we must have $\|\phi\|_{L^2}=1$, hence $\tilde{\phi}_n\rightarrow \phi$ strongly in $L^2(\R^2)$. In fact, $\tilde{\phi}_n \rightarrow \phi$ strongly in $L^r(\R^2)$ for $r\in [2,\infty)$, because of the $H^1_{x^\perp}(\R^2)$ boundedness. Thus we have
		$$
		E^{\rm NLS}_{1,a_*} \leq \Ec^{\rm NLS}_{1,a_*}(\phi) \leq \liminf_{n\rightarrow \infty} \Ec^{\rm NLS}_{1,a_n}(\phi_n) =\liminf_{n\rightarrow \infty} E^{\rm NLS}_{1,a_n}=E^{\rm NLS}_{1,a_*}.
		$$
		This shows that $\phi$ is a ground state for $E^{\rm NLS}_{1,a_*}$. However there are no such ground states, as proven in e.g.~\cite{Dinh-22,GuoLuoPen-21}, and we deduce that~\eqref{boun-phi-n} cannot hold. 
		
		We now conclude the proof by showing that $\{\phi_n\}_n$ blows up in $H^1(\R^2)$. We have
		\begin{align*}
		0 = E^{\rm NLS}_{1,a_*} = \lim_{n\rightarrow \infty} E^{\rm NLS}_{1,a_n} = \lim_{n\rightarrow \infty} \Ec^{\rm NLS}_{1,a_n}(\phi_n) =\lim_{n\rightarrow \infty} \|\nablax \phi_n\|^2_{L^2} - \frac{a_n}{2}\|\phi_n\|^4_{L^4}.
		\end{align*}
		Since $\|\nablax \phi_n\|_{L^2} \rightarrow \infty$ as $n\rightarrow \infty$, we must have $\|\phi_n\|^4_{L^4} \rightarrow \infty$. But then \eqref{GN-ineq} implies that $\|\nabla \phi_n\|_{L^2} \rightarrow \infty$ and $\|\nabla |\phi_n|\|_{L^2}\rightarrow \infty$ as well.
	\end{proof}
		
	We are now in the position to give the proof of Theorem \ref{theo-blow-nls}. 
	
	\begin{proof}[Proof of Theorem \ref{theo-blow-nls}]
		
		The proof is divided into several steps.
		
		\vspace{10px} \noindent {\bf Step 1. Convergence of the modulus.} We first show that there exists a sequence $\{x_n\}_n \subset \R^2$ such that
		\begin{align}\label{conv-modu}
		\vareps_n |\phi_n|(\vareps_n \cdot + x_n) \rightarrow Q_0 \text{ strongly in } H^1(\R^2) \text{ as } n\rightarrow \infty
		\end{align}
		where $\varepsilon_n$ is given by \eqref{length}.	Denote 
		$$v_n(x):= \vareps_n |\phi_n|(\vareps_n x).$$
		We then have
		\[
		\|v_n\|_{L^2}=\|\phi_n\|_{L^2}=1 \quad \text{and} \quad \|\nabla v_n\|_{L^2}= \vareps_n \|\nabla |\phi_n|\|_{L^2}=1.
		\]
		Hence $\{v_n\}_n$ is a bounded sequence in $H^1(\R^2)$. On the other hand, using \eqref{dia-ineq} we have
		\[
		\Ec^{\rm NLS}_{1,a}(\phi) \geq \|\nabla |\phi|\|^2_{L^2}-\frac{a}{2}\|\phi\|^4_{L^2}=:\Ec^0_a(|\phi|).
		\]
		But \eqref{GN-ineq} implies
		\[
		\Ec^0_a(|\phi|) \geq \left(1-\frac{a}{a_*}\right) \|\nabla |\phi|\|^2_{L^2}.
		\]
		From this and Lemma \ref{lem-limi-an}, we obtain
		\[
		0=\lim_{n\rightarrow \infty} E^{\rm NLS}_{1,a_n} = \lim_{n\rightarrow \infty} \Ec^{\rm NLS}_{1,a_n}(\phi_n) \geq \liminf_{n\rightarrow \infty} \Ec^0_{a_n}(|\phi_n|) \geq 0.
		\]
		In particular, we have $\Ec^0_{a_n}(v_n) = \vareps_n^2 \Ec^0_{a_n}(|\phi_n|) \rightarrow 0$ as $n\rightarrow \infty$. Since by definition 
		$$\|\nabla v_n\|_{L^2}=1$$
		for all $n\geq 1$, we infer that, up to a subsequence,
		\[
		\inf_{n\geq 1} \|v_n\|_{L^4} \geq C
		\]
		for some constant $C>0$. By Lemma \ref{lem-comp}, there exists $\phi \in H^1(\R^2) \backslash \{0\}$ and $\{y_n\}_n\subset \R^2$ such that up to a subsequence,
		\[
		\tilde{v}_n(x):=v_n(\cdot+y_n) \rightharpoonup \phi \text{ weakly in } H^1(\R^2).
		\]
		We next show that $\|\phi\|_{L^2}=1$. In fact, we first have
		\[
		0< \|\phi\|^2_{L^2} \leq \liminf_{n\rightarrow \infty} \|\tilde{v}_n\|^2_{L^2}=\lim_{n\rightarrow \infty} \|v_n\|_{L^2}=1.
		\]
		Assume for contradiction that $\|\phi\|_{L^2}<1$. Then by the weak convergence in $H^1$, we have as in \eqref{ener-decom} that
		\begin{align} \label{limi-E0-an}
		0 = \lim_{n\rightarrow \infty} \Ec^0_{a_n}(v_n) = \lim_{n\rightarrow \infty} \Ec^0_{a_n}(\tilde{v}_n) \geq \Ec^0_{a_*}(\phi) + \lim_{n\rightarrow \infty} \Ec^0_{a_*}(\tilde{v}_n-\phi).
		\end{align}
		Again, by \eqref{GN-ineq}, we have
		\[
		\Ec^0_{a_*} (\tilde{v}_n-\phi) \geq 0
		\]
		and
		\[
		\Ec^0_{a_*}(\phi) = \|\phi\|^2_{L^2} \Ec^0_{a_*} \left(\frac{\phi}{\|\phi\|_{L^2}}\right) + \frac{a_*}{2} \left(\frac{1}{\|\phi\|^2_{L^2}}-1\right) \|\phi\|^4_{L^4} >0
		\]
		since $0<\|\phi\|_{L^2}<1$. This is contradiction with~\eqref{limi-E0-an} and we thus must have $\|\phi\|_{L^2}=1$. Then $\tilde{v}_n\rightarrow \phi$ strongly in $L^2(\R^2)$, up to a subsequence. In fact, $\tilde{v}_n\rightarrow \phi$ strongly in $L^r(\R^2)$ for $r\in [2,\infty)$, because of the $H^1(\R^2)$ boundedness. Therefore,
		$$
		0 \leq \Ec^0_{a_*}(\phi) \leq \liminf_{n\rightarrow \infty}\Ec^0_{a_*}(\tilde{v}_n) = \liminf_{n\rightarrow \infty} \Ec^0_{a_n}(v_n) =0.
		$$
		This shows that 
		$$
		\lim_{n\rightarrow \infty} \|\nabla \tilde{v}_n\|_{L^2} = \lim_{n\rightarrow \infty} \frac{a_n}{2}\|\tilde{v}_n\|_{L^2} = \lim_{n\rightarrow \infty} \frac{a_*}{2}\|\phi\|_{L^2} = \|\nabla \phi\|_{L^2}.
		$$
		Hence $\tilde{v}_n\rightarrow \phi$ strongly in $H^1(\R^2)$, up to a subsequence. Moreover, $\phi$ is an optimizer of \eqref{GN-ineq}. By the uniqueness (up to translations and dilations) of optimizers for \eqref{GN-ineq} and the fact that $\tilde{v}_n$ is non-negative, there exist $\lambda>0$ and $x_0 \in \R^2$ such that $\phi(x) = \lambda Q_0(\lambda(x+x_0))$. Since 
		$\|\nabla \phi\|_{L^2}=1$, we must have $\lambda=1$. Again, by uniqueness of $Q_0$, we conclude that passing to a subsequence is unnecessary. This leads to \eqref{conv-modu} after setting $x_n= \vareps_n(y_n -x_0)$.

		\vspace{10px} \noindent {\bf Step 2. A property of Lagrange multipliers.} The minimizer $\phi_n$ of $E^{\rm NLS}_{1,a_n}$ satisfies the Euler--Lagrange equation 
		\begin{align} \label{mu-n}
		\(\nablax\)^2 \phi_n - a_n |\phi_n|^2 \phi_n = \mu_n \phi_n \quad \text{in} \quad \R^2
		\end{align}
		in the weak sense\footnote{Meaning
			\[
			\int_{\mathbb R^2}\nablax \phi_n \cdot \nablax \chi - a_n |\phi_n|^2\phi_n \chi - \mu_n \phi_n \chi \dx =0, \quad \forall \chi \in H^1_{x^\perp}(\R^2).
			\]}, where $\mu_n \in \R$ is the Lagrange multiplier. In this step, we show that $\vareps_n^2 \mu_n \rightarrow-1$ as $n\rightarrow \infty$. Indeed, as $\phi_n$ is a ground state for $E^{\rm NLS}_{1,a_n}$, using \eqref{mu-n}, we have
		\[
		\mu_n = \|\nablax \phi_n\|^2_{L^2} - a_n\|\phi_n\|^4_{L^4} = \Ec^{\rm NLS}_{1,a_n}(\phi_n) - \frac{a_n}{2} \|\phi_n\|^4_{L^4} = E^{\rm NLS}_{1,a_n}-\frac{a_n}{2} \|\phi_n\|^4_{L^4}.
		\]
		Denote
		\begin{align} \label{varphi-n}
		\varphi_n(x) = e^{i\theta_n}\psi_n(x)
		\end{align}
		with
		\[
		\psi_n(x):= \vareps_n \phi_n(\vareps_n x+x_n)e^{i\vareps_n x_n^\perp\cdot x}
		\]
		and $\theta_n \in [0,2\pi)$ satisfying
		\begin{align} \label{ortho-cond-Q0}
		\|\varphi_n - Q_0\|_{L^2} = \min_{\theta \in [0,2\pi)} \|e^{i\theta} \psi_n - Q_0\|_{L^2}.
		\end{align}
		By \eqref{conv-modu}, we have $|\varphi_n| := \vareps_n |\phi_n|(\vareps_n\cdot+x_n) \to Q_0$ strongly in $H^1(\R^2)$. Therefore,
		\[
		\lim_{n\rightarrow \infty} \vareps_n^2 \|\phi_n\|^4_{L^4} = \lim_{n\rightarrow \infty} \|\varphi_n\|^4_{L^4} = \|Q_0\|^4_{L^4} = \frac{2}{a_*}. 
		\]
		Since $0\leq E^{\rm NLS}_{1,a_n} \rightarrow 0$ (see Lemma \ref{lem-limi-an}) and $a_n \nearrow a_*$, we get
		\[
		\lim_{n\rightarrow \infty} \vareps_n^2 \mu_n = \lim_{n\rightarrow \infty} \vareps_n^2E^{\rm NLS}_{1,a_n} - \lim_{n\rightarrow \infty} \frac{a_n}{2} \vareps_n^2 \|\phi_n\|^4_{L^4} =-1.
		\]

		\vspace{10px} \noindent {\bf Step 3. A sub-equation for $|\varphi_n|^2$.} We next use \eqref{mu-n} to derive an equation and a sub-equation satisfied by $\varphi_n$ and $|\varphi_n|^2$. To do so, we write
		\[
		\psi_n(x) = \vareps_n \tilde{\phi}_n(\vareps_n x)
		\]
		with $\tilde{\phi}_n(x):= \phi_n(x+x_n)e^{ix_n^\perp \cdot x}$. A direct computation gives
		\[
		\(\nablax\)^2 \tilde{\phi}_n(x) = \(\nablax\)^2 \phi_n(x+x_n) e^{ix_n^\perp \cdot x}
		\]
		which, by \eqref{mu-n}, implies
		\[
		\(\nablax\)^2 \tilde{\phi}_n - a_n|\tilde{\phi}_n|^2 \tilde{\phi}_n = \mu_n\tilde{\phi}_n.
		\]
		Using the identity
		\[
		\(\nablax\)^2 \phi = -\Delta \phi + 2 L\phi + |x|^2 \phi
		\]
		with $L = i(x_2 \partial_1 - x_1 \partial_2) =-ix^\perp \cdot \nabla$, we see that $\tilde{\phi}_n$ solves the elliptic equation
		\[
		-\Delta \tilde{\phi}_n +|x|^2 \tilde{\phi}_n + 2 L\tilde{\phi}_n -a_n|\tilde{\phi}_n|^2 \tilde{\phi}_n -\mu_n\tilde{\phi}_n =0
		\]
		in the weak sense. By the definition of $\varphi_n$ in \eqref{varphi-n}, we get
		\begin{align}\label{equ-varphi-n}
		-\Delta \varphi_n + \vareps_n^4 |x|^2 \varphi_n + 2\vareps_n^2 L\varphi_n - a_n |\varphi_n|^2 \varphi_n -\vareps_n^2 \mu_n \varphi_n =0.
		\end{align}
		Denote $W_n := |\varphi_n|^2$. Multiplying both sides of \eqref{equ-varphi-n} with $\overline{\varphi}_n$, taking the real part, and using the following identities
		\begin{align} \label{iden-Wn}
		\begin{aligned}
		-\rea(\Delta \varphi_n \overline{\varphi}_n) &=-\frac{1}{2} \Delta W_n + |\nabla \varphi_n|^2, \\
		2\rea(L\varphi_n \overline{\varphi}_n) &= L\varphi_n \overline{\varphi}_n + \overline{L\varphi_n} \varphi_n = x^\perp \cdot J(\varphi_n),
		\end{aligned}
		\end{align}
		with $J(\varphi)= i(\varphi \nabla \overline{\varphi} - \overline{\varphi} \nabla \varphi)$ the superfluid current, we obtain
		\begin{align} \label{equ-Wn}
		-\frac{1}{2} \Delta W_n + |\nabla \varphi_n|^2 + \vareps_n^4|x|^2 W_n + \vareps^2_n x^\perp \cdot J(\varphi_n) - a_n W_n^2 - \vareps_n^2 \mu_n W_n =0.
		\end{align}
		Using the identity
		\begin{align*}
		|(-i \nabla + \vareps_n^2 x^\perp) \varphi_n|^2 = |\nabla \varphi_n|^2 + \vareps_n^2 x^\perp \cdot J(\varphi_n) + \vareps_n^4 |x|^2 W_n,
		\end{align*}
		we deduce that 
		\begin{align} \label{sub-equ-Wn}
		-\frac{1}{2}\Delta W_n - \vareps_n^2 \mu_n W_n - a_n W_n^2 \leq 0
		\end{align}
		in the weak sense, i.e., 
		\[
		\int_{\mathbb R^2}\frac{1}{2} \nabla W_n \cdot \nabla \chi - \vareps_n^2 \mu_n W_n \chi - a_n W_n^2 \chi \dx \leq 0, \quad \forall 0\leq \chi \in H^1_{x^\perp}(\R^2).
		\]

		\vspace{10px} \noindent {\bf Step 4. Uniform boundedness of $W_n$.} To prove the uniform boundedness of the sub-solution $W_{n} = |\varphi_n|^2$ to \eqref{sub-equ-Wn}, we need its local boundedness. The following formulation is taken from \cite[Theorem 4.14]{HanLin-11} (see \cite[Theorem 4.1]{HanLin-11} and \cite[Theorem 8.17]{GilTru-01} for the proof).
		
		\begin{theorem}[\textbf{Local boundedness}] \label{theo-DeG-Nas-Mos}\mbox{}\\
			Let $\Omega$ be a domain in $\R^d$. Assume that $a_{jk} \in L^\infty(\Omega)$ satisfies
			\[
			\lambda |\xi|^2 \leq \sum_{j,k} a_{jk}(x)\xi_j \xi_k \leq \Lambda |\xi|^2, \quad \forall x \in \Omega,~\forall \xi \in \R^d
			\]
			for some positive constants $\lambda$ and $\Lambda$. Let $u \in H^1(\Omega)$ be a non-negative sub-solution in $\Omega$ in the following sense
			\[
			\int_\Omega a_{jk} \partial_j u \partial_k \chi \dx \leq \int_\Omega f \chi \dx, \quad \forall \chi \in H^1_0(\Omega),~ \chi \geq 0 \text{ in } \Omega.
			\]
			Suppose that $f \in L^q(\Omega)$ for some $q>\frac{d}{2}$. Then there holds for any $B_R(x_0) \subset \Omega$ and any $p>0$
			\[
			\sup_{B_{R/2}(x_0)} u(x) \leq C \left( R^{-\frac{d}{p}} \|u\|_{L^p(B_R(x_0))} + R^{2-\frac{d}{q}} \|f\|_{L^q(B_R(x_0))}\right),
			\]
			where $C=C(d,\lambda, \Lambda, p,q)$ is a positive constant.
		\end{theorem}
		
		Let $M>0$ and denote $\Omega_M= \{x \in \R^2 : |x|\geq M\}$. Applying Theorem \ref{theo-DeG-Nas-Mos} to \eqref{sub-equ-Wn} with $\Omega=\Omega_M$, $a_{jk}=\frac{1}{2} \delta_{jk}$, $f=\vareps_n^2 \mu_n W_n+a_nW_n^2$, $p=q=2$, $R=2$, and $B_2(x_0) \subset \Omega_M$, we get
		\begin{align} \label{DeG-Nas-Mos-appl-1}
		\sup_{B_1(x_0)} W_n(x) \leq C \left( \|W_n\|_{L^2(B_2(x_0))} + \|W_n^2\|_{L^2(B_2(x_0))}\right)
		\end{align}
		for some universal constant $C>0$. Since $B_2(x_0) \subset \Omega_M$, we deduce
		\begin{align*}
		\|W_n\|_{L^2(B_2(x_0))} + \|W_n^2\|_{L^2(B_2(x_0))} &\leq \|W_n\|_{L^2(|x|\geq M)} + \|W_n^2\|_{L^2(|x|\geq M)} \\
		&\rightarrow \|Q_0^2\|_{L^2(|x|\geq M)} + \|Q_0^4\|_{L^2(|x|\geq M)}.
		\end{align*}
		Here we have used $\epsilon_{n}^2\mu_n \to -1$ and the fact that $W_n \rightarrow Q_0^2$ in $L^2(\R^2)$ and $W_n^2 \rightarrow Q_0^4$ in $L^2(\R^2)$ because
		\begin{align*}
		\|W_n-Q_0^2\|_{L^2} &\leq \||\varphi_n|-Q_0\|_{L^4} \||\varphi_n|+Q_0\|_{L^4}, \\
		\|W_n^2-Q_0^4\|_{L^2}&\leq \||\varphi_n|-Q_0\|_{L^8} \||\varphi_n|+Q_0\|_{L^8} \||\varphi_n|^2+Q_0^2\|_{L^4},
		\end{align*}
		and $|\varphi_n|\rightarrow Q_0$ strongly in $L^r(\R^2)$ for all $r\in [2,\infty)$. The later follows from the strong convergence $|\varphi_n| \to Q_0$ in $H^1(\R^2)$ and Sobolev embedding. In particular, for $\eps>0$, there exist $n_\eps \in \N$ and $M_\eps$ sufficiently large such that for all $n\geq n_\eps$ and all $M\geq M_\eps$,
		\[
		\|W_n\|_{L^2(B_2(x_0))} +\|W_n^2\|_{L^2(B_2(x_0))}\leq \frac{\eps}{C}
		\]
		which together with \eqref{DeG-Nas-Mos-appl-1} yield
		\begin{align*} 
		\sup_{B_1(x_0)} W_n(x) \leq \eps
		\end{align*}
		for all $B_1(x_0) \subset \Omega_{M_\eps}$. As $B_1(x_0)$ is arbitrarily in $\Omega_{M_\eps}$, we get (by possibly increasing $M_\eps$)
		\begin{align} \label{est-Wn-1}
		W_n(x) \leq \epsilon \text{ for all } |x| \geq M_\eps \text{ and all $n$ sufficiently large.}
		\end{align}
		Applying again Theorem \ref{theo-DeG-Nas-Mos} to \eqref{sub-equ-Wn} with $\Omega=\R^2$, $a_{jk}=\frac{1}{2} \delta_{jk}$, $f=\mu_n\vareps_n^2 W_n + a_n W_n^2$, $p=q=2$, and $R=2M_\eps$, we get
		\[
		\sup_{B_{M_\eps}(0)} W_n(x) \leq C \left( M_\eps^{-1}\|W_n\|_{L^2(B_{2M_\eps}(0))} + M_\eps \|W_n^2\|_{L^2(B_{2M_\eps}(0))}\right)
		\] 
		for some universal constant $C>0$. This implies
		\begin{align} \label{est-Wn-2}
		\sup_{B_{M_\eps}(0)} W_n(x) \leq C(M_\eps) \text{ for all $n$ sufficiently large.} 
		\end{align}
		Collecting \eqref{est-Wn-1} and \eqref{est-Wn-2}, we prove 
		\begin{align} \label{bound-Wn}
		0 \leq \sup_{x\in \R^2} W_n(x) \leq C \text{ for all $n$ sufficiently large},
		\end{align}
		where $C>0$ is a constant independent of $n$.

		\vspace{10px} \noindent {\bf Step 5. Uniform exponential decay of $W_n$.} We now prove the uniform exponential decay of $W_n$. To this end, we test \eqref{sub-equ-Wn} with
		$e^{\alpha|x|} W_n$ for some constant $\alpha$ to be chosen shortly. We have
		\[
		-\frac{1}{2} \int_{\mathbb R^2}\Delta W_n e^{\alpha |x|} W_n \dx - \mu_n\vareps_n^2 \int_{\mathbb R^2}W_n e^{\alpha|x|}W_n \dx - a_n \int_{\mathbb R^2}W_n^2 e^{\alpha |x|} W_n \dx \leq 0.
		\]
		Observe that
		\begin{align*}
		\int_{\mathbb R^2}\Delta W_n e^{\alpha|x|} W_n \dx &= \int_{\mathbb R^2}e^{\alpha |x|} \left( \frac{1}{2}\Delta (W_n^2) - |\nabla W_n|^2 \right) \dx \\
		&= \frac{1}{2} \int_{\mathbb R^2}W_n^2 \Delta(e^{\alpha|x|}) \dx - \int_{\mathbb R^2}|\nabla W_n|^2 e^{\alpha |x|} \dx \\
		&=\frac{1}{2}\int_{\mathbb R^2}W_n^2 \left(\alpha^2+\frac{\alpha}{|x|}\right) e^{\alpha |x|} \dx - \int_{\mathbb R^2}|\nabla W_n|^2 e^{\alpha|x|} \dx
		\end{align*}
		and
		\begin{align*}
		\int_{\mathbb R^2}|\nabla(W_n e^{\alpha|x|/2})|^2 \dx = \frac{\alpha^2}{4} \int_{\mathbb R^2}W_n^2 e^{\alpha|x|} \dx + \int_{\mathbb R^2}|\nabla W_n|^2 e^{\alpha|x|} \dx - \frac{1}{2}\int_{\mathbb R^2}W_n^2 \left(\alpha^2+ \frac{\alpha}{|x|}\right)e^{\alpha|x|} \dx.
		\end{align*}
		In particular, we have
		\[
		\int_{\mathbb R^2}\Delta W_n e^{\alpha|x|} W_n \dx = \frac{\alpha^2}{4} \int_{\mathbb R^2}W_n^2 e^{\alpha|x|} \dx - \int_{\mathbb R^2}| \nabla ( W_n e^{\alpha|x|/2} ) |^2 \dx,
		\]
		hence
		\[
		\frac{1}{2} \int_{\mathbb R^2}| \nabla ( W_n e^{\alpha|x|/2} ) |^2 \dx - \frac{\alpha^2}{8} \int_{\mathbb R^2}W_n^2 e^{\alpha|x|} \dx  -\mu_n\vareps_n^2 \int_{\mathbb R^2}W_n^2 e^{\alpha |x|} \dx - a_n\int_{\mathbb R^2}W_n^3 e^{\alpha |x|} \dx \leq 0
		\]
		so
		\[
		\int_{\mathbb R^2}\left(-\mu_n\vareps_n^2 -\frac{\alpha^2}{8} - a_n W_n\right)W_n^2 e^{\alpha |x|} \dx \leq 0.
		\]
		We pick $\alpha=1$ and choose $M>0$ so large that $W_n(x) \leq \frac{1}{4a_*}$ for all $|x|\geq M$ and all $n$ sufficiently large (see \eqref{est-Wn-1}). As $\mu_n\vareps_n^2\to -1$ (by Step 1), we get
		\[
		-\mu_n\vareps_n^2 -\frac{1}{8}-a_n W_n(x) \geq \frac{1}{2}
		\]
		for all $|x|\geq M$ and all $n$ sufficiently large. Thus we obtain
		\begin{align*}
		\frac{1}{2} \int_{\R^2 \backslash B_M(0)} W_n^2 e^{|x|} \dx \leq \int_{B_M(0)} \left|-\mu_n\vareps_n^2 -\frac{1}{8}-a_n W_n\right| W_n^2 e^{|x|} \dx \leq C e^M \|W_n\|^2_{L^2} \leq C e^M
		\end{align*}
		for all $n$ sufficiently large, where we have used \eqref{bound-Wn} to get the second estimate. This proves that
		\begin{align} \label{expo-deca-Wn}
		\int_{\R^2} W_n^2 e^{|x|} \dx \leq C,
		\end{align}
		for all $n$ sufficiently large, where $C>0$ is independent of $n$. From this, we get
		\begin{align} \label{expo-deca-varphi-n}
		\int_{\mathbb R^2}|\varphi_n|^2 e^{|x|/4} \dx = \int_{\mathbb R^2}W_n e^{|x|/2} e^{-|x|/4} \dx\leq \left(\int_{\mathbb R^2}W_n^2 e^{|x|} \dx\right)^{1/2} \left(\int_{\mathbb R^2}e^{-|x|/2} \dx\right)^{1/2} \leq C
		\end{align}
		for all $n$ sufficiently large. An immediate consequence of this uniform exponential decay is that $|x||\varphi_n|\rightarrow |x|Q_0$ strongly in $L^2(\R^2)$.

		\vspace{10px} \noindent {\bf Step 6. $H^1$-strong convergence.} By the definition of $\varphi_n$ (see \eqref{varphi-n}), we have
		\[
		\phi_n(x) = \vareps_n^{-1} \varphi_n(\vareps_n^{-1}(x-x_n)) e^{-i x_n^\perp \cdot x-i\theta_n}.
		\]
		Since $\phi_n$ is a ground state for $E^{\rm NLS}_{1,a_n}$, we see that
		\begin{align}\label{E-an-phi-n}
		E^{\rm NLS}_{1,a_n}=\Ec^{\rm NLS}_{1,a_n}(\phi_n) =\|\nabla \phi_n\|^2_{L^2} + \|x\phi_n\|^2_{L^2} + 2 \<L\phi_n,\phi_n\> - \frac{a_n}{2} \|\phi_n\|^4_{L^4}.
		\end{align}
		This implies the following identity
		\begin{align}\label{E-an-varphi-n}
		\vareps_n^2E^{\rm NLS}_{1,a_n} = \|\nabla \varphi_n\|^2_{L^2} +2 \vareps_n^2 \<L\varphi_n, \varphi_n\> + \vareps_n^4 \|x\varphi_n\|^2_{L^2} - \frac{a_n}{2} \|\varphi_n\|^4_{L^4}.
		\end{align}
		By the Cauchy--Schwarz inequality, we have 
		\[
		|2\vareps_n^2\<L\varphi_n, \varphi_n\>| \leq 2\vareps_n^2 \|\nabla \varphi_n\|_{L^2} \|x\varphi_n\|_{L^2} \leq \frac{1}{2} \|\nabla \varphi_n\|^2_{L^2} +2\vareps^4_n\|x\varphi_n\|^2_{L^2}
		\]
		which implies
		\[
		\|\nabla \varphi_n\|^2_{L^2} \leq 2 \left(\vareps_n^2 E^{\rm NLS}_{a_n} + \vareps_n^4 \|x\varphi_n\|^2_{L^2} + \frac{a_n}{2}\|\varphi_n\|^4_{L^4}\right).
		\]
		Since $E^{\rm NLS}_{1,a_n}\to 0$, $\vareps_n \to 0$, $|x||\varphi_n| \to |x|Q_0$ strongly in $L^2(\R^2)$, and $|\varphi_n| \to Q_0$ strongly in $L^4(\R^2)$, we infer that $\{\varphi_n\}_n$ is bounded uniformly in $H^1(\R^2)$. 
		
		From \eqref{E-an-varphi-n}, we also have
		\begin{align*} 
		\|\nabla \varphi_n\|^2_{L^2} - \frac{a_*}{2} \|\varphi_n\|^4_{L^4} = \vareps_n^2 E^{\rm NLS}_{1,a_n} - 2\vareps_n^2 \<L\varphi_n,\varphi_n\> - \vareps_n^4 \|x\varphi_n\|^2_{L^2} - \frac{a_*-a_n}{2}\|\varphi_n\|^4_{L^4}.
		\end{align*}
		Using the uniform boundedness of $\{\varphi_n\}_n$ in $H^1(\R^2)$, the strong convergence $|x||\varphi_n| \to |x| Q_0$ in $L^2(\R^2)$, and $a_n \nearrow a_*$, we deduce that
		\[
		\lim_{n\to \infty} \|\nabla \varphi_n\|^2_{L^2} -\frac{a_*}{2} \|\varphi_n\|^4_{L^4} =0.
		\]
		Since $\|\varphi_n\|_{L^2}=1$ and $|\varphi_n| \to Q_0$ strongly in $L^r(\R^2)$ for all $r\in [2,\infty)$, there exists $\{z_n\}_n\subset \R^2$ such that
		\begin{align} \label{zn}
		\varphi_n(x+z_n) \rightarrow e^{i\theta}Q_0(x)
		\end{align}
		strongly in $H^1(\R^2)$, for some $\theta \in \R$. Using the fact that $\|Q_0(\cdot+z_n) - Q_0\|_{H^1}\to 0$ if and only if $|z_n|\to 0$, we get $|z_n|\to 0$. This in turn implies that $\varphi_n \to e^{i\theta}Q_0$ strongly in $H^1(\R^2)$ since
		\begin{align*}
		\|\varphi_n-e^{i\theta}Q_0\|_{H^1} &=\|\varphi_n(\cdot+z_n) - e^{i\theta} Q_0(\cdot+z_n)\|_{H^1} \\
		&\leq \|\varphi_n(\cdot+z_n)- e^{i\theta} Q_0\|_{H^1} + \|Q_0 -Q_0(\cdot+z_n)\|_{H^1} \to 0.
		\end{align*}
		Now we write
		\[
		\varphi_n(x) = q_n(x) + i r_n(x)
		\]
		with $q_n$ and $r_n$ the real and imaginary parts of $\varphi_n$ respectively. By \eqref{ortho-cond-Q0}, we have the following orthogonality condition
		\begin{align} \label{ortho-cond}
		\int_{\R^2} Q_0 r_n \dx =0.
		\end{align}
		Since $\|\varphi_n - e^{i\theta} Q_0\|^2_{L^2} \to 0$, we have
		\[
		\int_{\mathbb R^2}\left(\rea(\varphi-e^{i\theta}Q_0)\right)^2+\left(\ima(\varphi_n-e^{i\theta}Q_0) \right)^2 \dx \to 0.
		\]
		In particular, we get
		\[
		\int_{\mathbb R^2}(r_n-Q_0 \sin \theta)^2 \dx \to 0.
		\]
		Using \eqref{ortho-cond}, we have
		\[
		\int_{\mathbb R^2}r_n^2 + Q_0^2 \sin^2\theta \dx \to 0.
		\]
		This shows that
		\[
		\int_{\mathbb R^2}r_n^2 \dx \to 0 \text{ and } \sin^2\theta = 0.
		\]
		Hence $\theta=0$ and $\varphi_n \to Q_0$ strongly in $H^1(\R^2)$. In particular, we have
		\begin{align} \label{qn-rn}
		\int_{\mathbb R^2}(q_n -Q_0)^2 \dx \to 0 \quad \text{and} \quad \int_{\mathbb R^2}r_n^2 \dx \to 0.
		\end{align}
		This, together with the exponential decay of $W_n$, yields
		\begin{align} \label{vari-qn-rn}
		\int_{\mathbb R^2}|x|^2(q_n-Q_0)^2 \dx \to 0 \quad \text{and} \quad \int_{\mathbb R^2}|x|^2 r_n^2 \dx \to 0.
		\end{align}
		In fact, by the exponential decay of $W_n$ (see \eqref{expo-deca-varphi-n}), we have
		\begin{align*}
		\int_{\mathbb R^2}|x|^2 r_n^2 \dx \leq \left(\int_{\mathbb R^2}|x|^4 r_n^2 \dx\right)^{1/2} \left(\int_{\mathbb R^2}r_n^2 \dx\right)^{1/2} \to 0
		\end{align*}
		and similarly for $q_n-Q_0$.

		\vspace{10px} \noindent {\bf Step 7. Smallness of the imaginary part.} Observe that
		\begin{align} \label{iden-L-varphi-n}
		\begin{aligned}
		\<L\varphi_n, \varphi_n\> = \rea \<L\varphi_n, \varphi_n\> &= \int_{\mathbb R^2}x^\perp \cdot \ima(\overline{\varphi}_n \nabla \varphi_n) \dx \\
		&= \int_{\mathbb R^2}x^\perp \cdot (q_n \nabla r_n - r_n \nabla q_n) \dx = 2 \int_{\mathbb R^2}x^\perp q_n \nabla r_n \dx
		\end{aligned}
		\end{align}
		which implies
		\[
		|\<L\varphi_n, \varphi_n\>| \leq 2\|xq_n\|_{L^2} \|\nabla r_n\|_{L^2} \leq C \|\nabla r_n\|_{L^2}.
		\]
		Here we have used the fact that $|x| q_n$ is bounded uniformly in $L^2(\R^2)$ since $|x||\varphi_n| \rightarrow |x| Q_0$ strongly in $L^2(\R^2)$. We deduce from the above and \eqref{E-an-varphi-n} that
		\begin{align*}
		\vareps_n^2 E^{\rm NLS}_{1,a_n} \geq \int_{\mathbb R^2}|\nabla q_n|^2 +|\nabla r_n|^2 -\frac{a_*}{2} (q_n^4 + r_n^4 + 2q_n^2 r_n^2) \dx - C \vareps_n^2 \|\nabla r_n\|_{L^2}.
		\end{align*}
		We have
		\begin{align*}
		\frac{a_*}{2} \int_{\mathbb R^2}(r_n^4+2q_n^2 r_n^2) \dx \leq a_* \int_{\mathbb R^2}|\varphi_n|^2 r_n^2 \dx & = \int_{\mathbb R^2}Q^2 r_n^2 \dx + a_* \int(|\varphi_n|^2-Q_0^2) r_n^2 \dx \\ & = \int_{\mathbb R^2}Q^2 r_n^2 \dx + o(1)\|r_n\|^2_{H^1}.
		\end{align*}
		Here we have used that
		\begin{align*}
		\left|\int_{\mathbb R^2}(|\varphi_n|^2-Q_0^2)r_n^2 \dx\right| \leq \||\varphi_n|^2-Q_0^2\|_{L^2} \|r_n\|^2_{L^4} \leq C\||\varphi_n|^2-Q_0^2\|_{L^2} \|r_n\|^2_{H^1}
		\end{align*}
		and 
		\[
		\||\varphi_n|^2-Q_0^2\|_{L^2} \leq \||\varphi_n|-Q_0\|_{L^4} \||\varphi_n|+Q_0\|_{L^4} \to 0
		\]
		as $|\varphi_n| \to Q_0$ strongly in $H^1(\R^2)$ hence in $L^4(\R^2)$ by Sobolev embedding. On the other hand, by \eqref{GN-ineq}, we have
		\[
		\int_{\mathbb R^2}|\nabla q_n|^2 - \frac{a_*}{2} q_n^4 \dx \geq \|\nabla q_n\|^2_{L^2}(1-\|q_n\|^2_{L^2}) = (1+o(1)) \|r_n\|^2_{L^2},
		\]
		where we have used that $q_n \to Q_0$ strongly in $H^1(\R^2)$, $\|q_n\|^2_{L^2} + \|r_n\|^2_{L^2}=1$ as $\|\varphi_n\|^2_{L^2}=1$, and $\|\nabla Q_0\|^2_{L^2}=1$. Thus we get
		\begin{align*}
		\vareps_n^2 E^{\rm NLS}_{1,a_n} &\geq \int_{\mathbb R^2}|\nabla r_n|^2 - Q^2 r_n^2 +r_n^2 \dx + o(1) \|r_n\|^2_{H^1} - C\vareps_n^2 \|\nabla r_n\|_{L^2} \\
		&= \<\Lc r_n, r_n\> +o(1) \|r_n\|^2_{H^1} - C\vareps_n^2 \|\nabla r_n\|_{L^2},
		\end{align*}
		where $\Lc := -\Delta - Q^2 +1$. 
		
		We now use the non-degeneracy property of $Q$. It is well-known (see \cite[Theorem 11.8 and Corrollary 11.9]{LieLos-01}) that $Q$ is the first eigenfunction of $\Lc$ and the corresponding eigenvalue $0$ is non-degenerate. In particular, we have
		\[
		\<\Lc u, u\> \geq \lambda_2 \|u\|^2_{L^2}
		\]
		for all $u$ orthogonal to $Q$, where $\lambda_2>0$ is the second eigenvalue of $\Lc$. This together with the fact that
		\[
		\<\Lc u, u\> \geq \|u\|^2_{H^1} -\|Q\|^2_{L^\infty} \|u\|^2_{L^2}
		\]
		yield
		\[
		\<\Lc u, u\> \geq C \|u\|^2_{H^1}
		\]
		for some constant $C>0$ and all $u$ orthogonal to $Q$. Thanks to this estimate and the orthogonality condition \eqref{ortho-cond}, we get
		\[
		\vareps_n^2 E^{\rm NLS}_{1,a_n} \geq C_1 \|r_n\|^2_{H^1} - C_2\vareps_n^2 \|\nabla r_n\|_{L^2}
		\]
		for some positive constants $C_1$ and $C_2$. This implies that
		\begin{align} \label{est-rn-1} 
		\|r_n\|^2_{H^1} \leq C(\vareps_n^2 E^{\rm NLS}_{1,a_n} + \vareps_n^4).
		\end{align}
		On the other hand, from \eqref{uppe-boun-Ea}, \eqref{mag-GN-ineq} and \eqref{dia-ineq}, we have
		\begin{align*}
		C (a_*-a_n)^{1/2} \geq E^{\rm NLS}_{1,a_n} = \Ec^{\rm NLS}_{1,a_n} (\phi_n) \geq \frac{a_*-a_n}{a_*} \|\nablax \phi_n\|^2_{L^2} \geq \frac{a_*-a_n}{a_*} \|\nabla |\phi_n|\|^2_{L^2} = \frac{a_*-a_n}{a_*} \vareps_n^{-2}
		\end{align*}
		which implies
		\begin{align}\label{est-vareps-n}
		E^{\rm NLS}_{1,a_n}\leq C(a_*-a_n)^{1/2} \leq C \vareps_n^2
		\end{align}
		for some constant $C>0$. This together with \eqref{est-rn-1} yield
		\begin{align} \label{est-rn}
		\|r_n\|_{H^1} \leq C \vareps_n^2.
		\end{align}

		\vspace{10px} \noindent {\bf Step 8. Identifying the blow-up limit.} Coming back to \eqref{iden-L-varphi-n}, we have
		\begin{align*}
		\<L\varphi_n, \varphi_n\> =2 \int_{\mathbb R^2}x^\perp \cdot \nabla r_n q_n \dx &= 2\int_{\mathbb R^2}x^\perp \cdot \nabla r_n Q_0 \dx + 2 \int_{\mathbb R^2}x^\perp \cdot \nabla r_n (q_n-Q_0) \dx \\
		&= 2 \int_{\mathbb R^2}x^\perp \cdot \nabla r_n (q_n-Q_0) \dx
		\end{align*}
		where we have used the fact that $x^\perp \cdot \nabla Q_0=0$ since $Q_0$ is radial and \eqref{vari-qn-rn}. This shows that
		\begin{align} \label{est-L-varphi-n}
		|\<L\varphi_n, \varphi_n\>| \leq \|\nabla r_n\|_{L^2} \|x(q_n-Q_0)\|_{L^2} \leq o(1)\|\nabla r_n\|_{L^2} \leq o(1) \vareps_n^2.
		\end{align}
		Here we have used \eqref{est-rn} in the last inequality.
		
		From \eqref{E-an-phi-n} and \eqref{GN-ineq}, we have
		\begin{align*}
		E^{\rm NLS}_{1,a_n} \geq 2\<L\phi_n,\phi_n\> + \|x\phi_n\|^2_{L^2} = 2\<L\varphi_n,\varphi_n\> +\vareps_n^2 \|x\varphi_n\|^2_{L^2}.
		\end{align*}
		Denote
		\[
		\beta_n := \frac{\vareps_n}{(a_*-a_n)^{1/4}}.
		\]
		From \eqref{est-vareps-n}, we have
		\[
		\beta_n^2 \geq C>0.
		\]
		Moreover, using \eqref{uppe-boun-Ea}, we also have
		\[
		C\geq \frac{E^{\rm NLS}_{1,a_n}}{(a_*-a_n)^{1/2}} \geq \frac{2}{(a_*-a_n)^{1/2}}\<L\varphi_n,\varphi_n\> +\beta_n^2 \|x\varphi_n\|^2_{L^2}.
		\]
		Thanks to \eqref{est-L-varphi-n} and the fact that $|x||\varphi_n| \rightarrow |x| Q_0$ strongly in $L^2(\R^2)$, we deduce
		\begin{align*}
		C\geq \beta_n^2(\|xQ_0\|^2_{L^2} + o(1)).
		\end{align*}
		In particular, we deduce that $\{\beta_n\}_n$ is bounded away from zero. Passing to subsequence, we have $\beta_n \rightarrow \beta>0$ as $n\to \infty$. 
		
		By \eqref{E-an-varphi-n}, we have
		\begin{align*}
		E^{\rm NLS}_{1,a_n} & \geq \frac{a_*-a_n}{2}\|\phi_n\|^4_{L^4} + 2\<L\phi_n, \phi_n\> + \|x\varphi_n\|^2_{L^2} \\
		& = \frac{(a_*-a_n)^{1/2}}{2\beta_n^2}\|\varphi_n\|^4_{L^4} + 2\<L\varphi_n, \varphi_n\> + (a_*-a_n)^{1/2}\beta_n^2 \|x\varphi_n\|^2_{L^2}.
		\end{align*}
		Since $\varphi_n \rightarrow Q_0$ strongly in $H^1(\R^2)$, $|x| |\varphi_n| \rightarrow |x| Q_0$ strongly in $L^2(\R^2)$, and \eqref{est-L-varphi-n}, we infer that
		\begin{align*}
		\frac{E^{\rm NLS}_{1,a_n}}{(a_*-a_n)^{1/2}} \geq \frac{1}{2\beta^2} \|Q_0\|^4_{L^4} + \beta^2 \|xQ_0\|^2_{L^2} + o(1).
		\end{align*}
		Optimizing over $\beta>0$ and noticing that $\|Q_0\|_{L^4}^4 = \frac{2}{a_*}$ we get
		\begin{align}\label{liminf}
		\liminf_{n\rightarrow \infty} \frac{E^{\rm NLS}_{1,a_n}}{(a_*-a_n)^{1/2}} \geq 2\frac{\|xQ_0\|_{L^2}}{a_*^{1/2}} \quad \text{and} \quad \beta = \frac{1}{a_*^{1/4}\|xQ_0\|_{L^2}^{1/2}}.
		\end{align}
		From \eqref{limsup} and \eqref{liminf}, we obtain \eqref{nls-energy} and \eqref{nls-gs-critical}.

		\vspace{10px} \noindent {\bf Step 9. $L^\infty$ convergence.} We finally upgrade the $H^1$-convergence of $\{\varphi_n\}_n$ to $L^\infty$-convergence. To this end, we first show the uniform exponential decay for $\nabla \varphi_n$, namely
		\begin{align} \label{expo-deca-grad}
		\int_{\mathbb R^2}|\nabla \varphi_n|^2 e^{|x|/4} \dx \leq C
		\end{align}
		for all $n$ sufficiently large.	We multiply both sides of \eqref{equ-varphi-n} with $e^{\alpha|x|} \overline{\varphi}_n$, integrate over $\R^2$, and take the real part to get
		\[
		\rea \int_{\mathbb R^2}-\Delta \varphi_n e^{\alpha|x|} \overline{\varphi}_n + \vareps_n^4 |x|^2 e^{\alpha |x|} |\varphi_n|^2 + 2\vareps_n^2 L\varphi_n e^{\alpha|x|} \overline{\varphi}_n  - a_n |\varphi_n|^4 e^{\alpha |x|} - \vareps_n^2 \mu_n |\varphi_n|^2 e^{\alpha|x|} \dx =0.
		\]
		Arguing as in \cite[Lemma 3.2]{LewNamRou-17-proc}, we have
		\[
		\rea \int_{\mathbb R^2}-\Delta \varphi_n e^{\alpha|x|} \overline{\varphi}_n \dx = \int_{\mathbb R^2}|\nabla (e^{\alpha|x|/2} \varphi_n)|^2 \dx - \frac{\alpha^2}{2} \int_{\mathbb R^2}e^{\alpha|x|} |\varphi_n|^2 \dx.
		\]
		In particular, we get
		\begin{align*}
		0= \int_{\mathbb R^2}|\nabla(e^{\alpha|x|/2}\varphi_n)|^2 \dx + \vareps_n^4 \int_{\mathbb R^2}|x|^2 e^{\alpha |x|} |\varphi_n|^2 \dx &+ \int_{\mathbb R^2}e^{\alpha|x|} \Big(-a_n |\varphi_n|^2 - \vareps_n^2 \mu_n -\frac{\alpha^2}{4}\Big) |\varphi_n|^2 \dx \\
		&+2\vareps_n^2 \int_{\mathbb R^2}L\varphi_n e^{\alpha|x|} \overline{\varphi}_n \dx.
		\end{align*}
		Since $L(e^{\alpha|x|/2}) =0$, we have
		\begin{align*}
		\left|2\vareps_n^2 \int_{\mathbb R^2}L\varphi_n e^{\alpha|x|} \overline{\varphi}_n \dx\right| &= \left|2\vareps_n^2 \int_{\mathbb R^2}e^{\alpha|x|/2}  \overline{\varphi}_n L(e^{\alpha|x|/2} \varphi_n) \dx\right| \\
		&\leq 2\vareps_n^2 \|x^\perp e^{\alpha|x|/2} \varphi_n\|_{L^2} \|\nabla(e^{\alpha|x|/2} \varphi_n)\|_{L^2} \\
		&\leq \frac{1}{2} \int_{\mathbb R^2}|\nabla(e^{\alpha|x|/2}\varphi_n)|^2 \dx + 2\vareps_n^4 \int_{\mathbb R^2}|x|^2 e^{\alpha|x|} |\varphi_n|^2 \dx.
		\end{align*}
		It follows that
		$$
		\frac{1}{2}\int_{\mathbb R^2}|\nabla(e^{\alpha|x|/2}\varphi_n)|^2 \dx \leq \vareps_n^4 \int_{\mathbb R^2}|x|^2 e^{\alpha|x|} |\varphi_n|^2 \dx + \int_{\mathbb R^2}e^{\alpha|x|} \Big(a_n |\varphi_n|^2 + |\vareps_n^2 \mu_n| +\frac{\alpha^2}{4}\Big) |\varphi_n|^2 \dx
		$$
		By choosing $\alpha=\frac{1}{4}$, using \eqref{expo-deca-Wn}, \eqref{expo-deca-varphi-n} and the fact that $\vareps_n^2 \mu_n \to-1$, we obtain
		\begin{align}\label{expo-deca-nabla-varphi-n}
		\int_{\mathbb R^2}|\nabla(e^{|x|/8}\varphi_n)|^2 \dx \leq C
		\end{align}
		for all $n$ sufficiently large. Note that, by the triangle inequality,
		$$
		\|\nabla(e^{|x|/8}\varphi_n)\|_{L^2} = \left\|e^{|x|/8}\nabla\varphi_n + \frac{x}{8|x|}e^{|x|/8}\varphi_n\right\|_{L^2} \geq \|e^{|x|/8}\nabla\varphi_n\|_{L^2} - \frac{1}{8}\|e^{|x|/8}\varphi_n\|_{L^2}.
		$$
		Then the claim \eqref{expo-deca-grad} follows directly from \eqref{expo-deca-nabla-varphi-n} and \eqref{expo-deca-varphi-n}.
		
		We next show that $\{\varphi_n\}_n$ is bounded uniformly in $H^2(\R^2)$. To see this, we rewrite \eqref{equ-varphi-n} as 
		\[
		-\Delta \varphi_n + \varphi_n = (1+\vareps_n^2 \mu_n) \varphi_n - \vareps_n^4 |x|^2 \varphi_n -2 \vareps_n^2 L\varphi_n + a_n |\varphi_n|^2\varphi_n.
		\]
		Since $\{\varphi_n\}_n$ is bounded uniformly in $H^1(\R^2)$, the uniform exponential decay in \eqref{expo-deca-varphi-n} and \eqref{expo-deca-grad} imply that the right hand side is bounded uniformly in $L^2(\R^2)$. This shows that $\{\varphi_n\}_n$ is bounded uniformly in $H^2(\R^2)$. By the Sobolev embedding $H^{2}(\R^2) \subset L^\infty(\R^2)$ and the strong convergence $\varphi_n \to Q_0$ in $H^1(\R^2)$, we have that $\varphi_n$ converges strongly to $Q_0$ in $L^\infty(\R^2)$ and hence \eqref{nls-gs-critical}.
		\end{proof}
		
		\subsection{Collapse with an almost critical speed}
		We now study the blow-up behavior of minimizers for $E_{\Omega,a}$ when both $\Omega \nearrow 1$ and $a \nearrow a_*$ at the same time. To this end, we recall the following energy asymptotic formula when $\Omega=0$ (see \cite{GuoSei-13}):
		\begin{align} \label{ener-0}
		E^{\rm NLS}_{0,a} = \sqrt{a_*-a} \( 2\frac{\|xQ_0\|_{L^2}}{a_*^{1/2}} + o(1)\) \text{ as } a \nearrow a_*.
		\end{align}
		
		\begin{proof}[Proof of Corollary \ref{coro-blow-nls}]
		Let $\Omega_n \nearrow 1$, $a_n \nearrow a_*$ as $n\rightarrow \infty$, and $\phi_n$ be a minimizer for $E_{\Omega_n,a_n}$. We rewrite the energy functional as follows
		\begin{align}
		E^{\rm NLS}_{\Omega_n,a_n} = \Ec^{\rm NLS}_{\Omega_n, a_n}(\phi_n) &= \Omega_n\Ec^{\rm NLS}_{1, a_n}(\phi_n) + (1-\Omega_n)\Ec^{\rm NLS}_{0, a_n}(\phi_n) \nonumber \\
		&\geq \Omega_n E^{\rm NLS}_{1,a_n} + (1-\Omega_n) \Ec^{\rm NLS}_{0,a_n}(\phi_n) \label{deco-ener-1} \\
		&\geq \Omega_n E^{\rm NLS}_{1,a_n} + (1-\Omega_n) E^{\rm NLS}_{0,a_n}, \nonumber
		\end{align}
		where we have used that $\Ec^{\rm NLS}_{1, a_n}(\phi_n) \geq E^{\rm NLS}_{1, a_n}$ and $\Ec^{\rm NLS}_{0, a_n}(\phi_n) \geq E^{\rm NLS}_{0, a_n}$. Since both $E^{\rm NLS}_{1, a_n}$ and $E^{\rm NLS}_{0, a_n}$ have the same asymptotic formula (see \eqref{nls-energy} and \eqref{ener-0}), we obtain
		\[
		E^{\rm NLS}_{\Omega_n, a_n} = (a_*-a_n)^{1/2}\left(2\frac{\|xQ_0\|_{L^2}}{a_*^{1/2}} +o(1)\right).
		\]
		Let $\psi_n$ be a ground state for $E^{\rm NLS}_{1, a_n}$. By Theorem \ref{theo-blow-nls}, there exist sequences $\{x_n\}_n \subset \R^2$ and $(\vartheta_n)_n \subset [0,2\pi)$ such that
		\[
		\varphi_n(x) := \vareps_n \psi_n(\vareps_n x + x_n) e^{i\vareps_n x_n^\perp \cdot x+i\vartheta_n} \to Q_0(x)
		\]
		strongly in $H^1 \cap L^\infty(\R^2)$ as $n \to \infty$. We choose $\tilde{\psi}_{n}(x) := \psi_n(x+x_n) e^{i x_n^\perp \cdot x+i\vartheta_n}$ as a trial state for $E^{\rm NLS}_{\Omega_n,a_n}$ and obtain
		\begin{align}\label{deco-ener-2}
		E^{\rm NLS}_{\Omega_n,a_n} \leq \Ec^{\rm NLS}_{\Omega_n, a_n}(\tilde{\psi}_{n}) &= \Omega_n \Ec^{\rm NLS}_{1, a_n}(\tilde{\psi}_{n}) + (1-\Omega_n)\Ec^{\rm NLS}_{0, a_n}(\tilde{\psi}_{n}) \nonumber\\
		&= \Omega_n E^{\rm NLS}_{1, a_n} + (1-\Omega_n)\Ec^{\rm NLS}_{0, a_n}(\tilde{\psi}_{n}).
		\end{align}
		Here we have used the magnetic translation invariance of the energy functional $\Ec^{\rm NLS}_{1, a_n}$. Putting together \eqref{deco-ener-1} and \eqref{deco-ener-2}, we obtain
		\[
		\Ec^{\rm NLS}_{0, a_n}(\phi_n) \leq \Ec^{\rm NLS}_{0, a_n}(\tilde{\psi}_n).
		\]
		By \eqref{uppe-boun-Ea} and the arguments in the proof of Theorem \ref{theo-blow-nls} (especially \eqref{est-L-varphi-n}), we have
		\[
		\Ec^{\rm NLS}_{0, a_n}(\tilde{\psi}_n) = \Ec^{\rm NLS}_{1, a_n}(\tilde{\psi}_n) - 2\< \tilde{\psi}_n, L\tilde{\psi}_n \> = E^{\rm NLS}_{1, a_n} - 2\< \varphi_n, L\varphi_n \> \leq (a_*-a_n)^{1/2}\left(2\frac{\|xQ_0\|_{L^2}}{a_*^{1/2}} +o(1)\right).
		\]
		This together with \eqref{ener-0} show that $\phi_n$ is an approximate ground state for $E^{\rm NLS}_{0, a_n}$. We then conclude (see e.g., \cite[Step 5 in Section 3]{LewNamRou-17-proc}) that there exists a sequence of phases $\{\theta_n\}_n \subset [0,2\pi)$ such that
		\begin{align}\label{approx_GS}
		\lim_{n \to \infty}\frac{(a_*-a_n)^{1/4}}{a_*^{1/4}\|xQ_0\|_{L^2}^{1/2}} \phi_n \(\frac{(a_*-a_n)^{1/4}}{a_*^{1/4}\|xQ_0\|_{L^2}^{1/2}} x\) e^{i\theta_n} = Q_0(x)
		\end{align}
		strongly in $H^1(\R^2)$. In fact, we obtain the strong convergence in $L^\infty(\mathbb R^2)$, by the same arguments as in the proof of \eqref{nls-gs-critical}. 
		\end{proof}

	\section{Collapse of many-body ground states}
	\label{S3}
	\setcounter{equation}{0}
	
	In this section, we prove the large-$N$ behavior of ground states for \eqref{eq:GS ener many} given in Theorem \ref{thm:many-body}. 
	\begin{proof}[Proof of Theorem \ref{thm:many-body}]
	
		Following arguments from \cite{LewNamRou-17-proc}, we have
		\[
		CN^{-\beta}\|\nabla Q_N\|_{L^2}\|Q_N\|_{L^6}^3 + E^{\rm NLS}_{\Omega,a_N} \geq E^{\rm QM}_{\Omega,a_N}(N) \geq E^{\rm NLS}_{\Omega,a_N} - CN^{2\beta-1}.
		\]
		where $Q_N$ is given in Theorem \ref{thm:many-body}. Note that the above energy estimates as well as the asymptotic formula of $E^{\rm NLS}_{\Omega,a_N}$ are independent of $\Omega$. Therefore, we obtain \eqref{eq:CV_energy} for every $0\leq\Omega\leq 1$.
		
		To prove convergence of ground states as $\Omega = \Omega_N \nearrow 1$ we consider the perturbed Hamiltonian
		\begin{align}\label{HN eta}
		H_{a_N,N,\eta_N} = H_{a_N,N} + \eta_N\sum_{j=1}^N A_j
		\end{align}
		with ground-state energy per particle denoted $E^{\rm QM}_{\Omega_N,a_N,\eta_N}(N)$. Here $\eta_N >0$ is a small parameter to be chosen later and $A$ is a bounded self-adjoint operator on $L^2 (\mathbb R^2)$. The associated NLS energy functional is
		\begin{align*}
		\Ec^{\rm NLS}_{\Omega_N,a_N,\eta_N}(u) = \Ec^{\rm NLS}_{\Omega_N,a_N}(u) + \eta_N \< Au, u \>.
		\end{align*}
		Denote by $E^{\rm NLS}_{\Omega_N,a_N,\eta_N}$ the corresponding ground-state energy and $u_{\eta_N}$ its ground state. Let $\Phi_N$ be a ground state for $H_N = H_{N,0}$ and $\gamma_{\Phi_N} ^{(1)}$ its one-body reduced density matrix. As in \cite[Step~2 in Section~4]{LewNamRou-17-proc} we obtain
		\begin{align}\label{variation}
		\eta_N \Tr\left[ A \gamma_{\Phi_N} ^{(1)} \right] \geq \eta_N \left\langle u_{\eta_N} | A | u_{\eta_N} \right\rangle +  O (N^{2\beta -1}) + O (N^{3\alpha/4 - \beta}).
		\end{align}
		Again the above estimate is independent of $\Omega_N$. Under the assumption that $a_* - a_N = N ^{-\alpha}$ with
		\[
		0 < \alpha < \min \left\{\frac{4\beta}{5}, 2(1-2\beta) \right\}
		\]
		one can chose $\eta_N = N^{-\alpha/2 - \sigma}$ with
		\[
		0 < \sigma < \min\left\{1-2\beta - \frac{\alpha}{2},\beta - \frac{5\alpha}{4}\right\}
		\]
		in such a way that
		\[
		\eta_N = o\big(E^{\rm NLS}_{0,a_N}\big) = o\big( (a_* - a_N)^{1/2}\big) = o\big(N ^{-\alpha/2}\big)
		\]
		and also 
		\[ 
		\eta_N^{-1} N^{2\beta -1} + \eta_N^{-1} N^{3\alpha/4 - \beta} \underset{N\to\infty}{\longrightarrow} 0.
		\]
		Then dividing \eqref{variation} by $\eta_N$ and repeating the argument with $A$ changed to $-A$ yield
		\begin{equation}\label{almost}
		\left\langle u_{\eta_N} |A | u_{\eta_N} \right\rangle + o (1)\leq \Tr\left[ A \gamma_{\Phi_N} ^{(1)}\right] \leq  \left\langle u_{-\eta_N} |A| u_{-\eta_N }\right\rangle + o (1). 
		\end{equation}
		On the other hand, with the above choice of $\eta_N$, we have
		\[
		\Ec^{\rm NLS}_{\Omega_N,a_N}(u_{\eta_N}) = \Ec^{\rm NLS}_{\Omega_N,a_N,\eta_N}(u_{\eta_N}) + O(\eta_N \|A\|) \leq \Ec^{\rm NLS}_{\Omega_N,a_N}(u_0) + O(\eta_N \|A\|) = E^{\rm NLS}_{\Omega_N,a_N} + O(\eta_N \|A\|).
		\]
		By the argument in the proof of \eqref{nls-gs}, the above implies that
		\[
		\Ec^{\rm NLS}_{0,a_N}(u_{\eta_N}) \leq (a_*-a_n)^{1/2}\left(2\frac{\|xQ_0\|_{L^2}}{a_*^{1/2}} +o_N(1)\right) + O\left(\frac{\eta_N}{1-\Omega_N}\|A\|\right).
		\]
		It then follows that $(u_{\eta_N})$ and $(u_{-\eta_N})$ are sequences of quasi-ground states for $E^{\rm NLS}_{0,a_N}$, under the assumption on $\Omega_N$ in Theorem \ref{thm:many-body}. Thus both sequences satisfy \eqref{approx_GS}. Combining with \eqref{almost}, we get, after a dilation of space variables, trace-class weak-$\star$ convergence of $\gamma_{\Phi_N} ^{(1)}$ to $|Q_N\rangle \langle Q_N|$. Since no mass is lost in the limit, this convergence must hold in trace-class norm. The limit being rank $1$, this implies the convergence of higher order density matrices to tensor powers of the limiting operator by well-known arguments (recalled e.g. in~\cite[Section~2.2]{Rougerie-EMS}).
	\end{proof}

\bibliographystyle{acm}
\bibliography{biblio}
\end{document}